\documentclass[11pt]{amsart}
\usepackage{amsfonts,amsmath,amstext,amsbsy,amsthm,amssymb,graphics,graphicx}
\usepackage[utf8]{inputenc}
\usepackage[T1]{fontenc}
\usepackage{dsfont} 
\usepackage[
	colorlinks=true
	,breaklinks
	]{hyperref} 
\usepackage[hyphenbreaks]{breakurl} 
\usepackage{xcolor}
\definecolor{c1}{rgb}{0,0,1} 
\definecolor{c2}{rgb}{0,0.3,0.9} 
\definecolor{c3}{rgb}{0.3,0,0.9} 
\hypersetup{
    linkcolor={c1}, 
    citecolor={c2},
    urlcolor={c3}  
}
\usepackage[resetlabels,labeled]{multibib}
\usepackage{graphicx}
\usepackage{enumerate}   
\usepackage{wrapfig}
\makeindex
\usepackage[top=3cm, bottom=3cm,left=3cm,right=3cm]{geometry} 
\newtheorem{theorem}{Theorem}
\newtheorem{proposition}[theorem]{Proposition}
\newtheorem{lemma}[theorem]{Lemma} 

\newtheorem{definition}[theorem]{Definition}

\newcommand{\Dim}{\ensuremath{\textrm{Dim}_H}}
\newcommand{\diam}{\mbox{diam}}

\begin{document}


	\title[A Uniform Result for the Dimension of Level Sets]{\textbf{A Uniform Result for the Dimension of Fractional Brownian Motion Level Sets}}
	\author{LARA DAW}
	\date{\today}
\maketitle
	\begin{abstract}
Let $B =\{ B_t \, : \, t \geq 0 \}$ be a real-valued fractional Brownian motion of index $H \in (0,1)$. We prove that the macroscopic Hausdorff dimension of the level sets $\mathcal{L}_x = \left\{ t \in \mathbb{R}_+ \, : \, B_t=x \right\}$ is, with probability one, equal to $1-H$ for all $x\in\mathbb{R}$.\\
\\ 
\textit{Keywords:} Level sets; Fractional Brownian motion; Local times; Macroscopic Hausdorff dimension.  
\end{abstract}

\section{Introduction}

Let $B =\{ B_t \, : \, t \geq 0 \}$ be a fractional Brownian motion of index $H \in (0,1)$, that is, a centered, real-valued Gaussian process with covariance function
\begin{align}
R(s,t)= \mathbb{E} \left( B_s B_t\right)= \dfrac{1}{2}\left( \lvert s \rvert ^{2H} + \lvert t \rvert ^{2H} - \lvert s - t\rvert ^{2H}\right).
\end{align}
Since $\mathbb{E} \big[\left( B_s - B_t\right)^2\big]= \left|s - t\right|^{2H}$, it is an immediate consequence of the Kolmogorov–Centsov continuity theorem that $B$ admits a continuous modification. Throughout this note, we will always assume that $B$ is continuous. It is also immediate (see, e.g., \cite{nourdin2012selected}) that $B$ is a self-similar process of exponent $H$,
that is, for any $a > 0$,
\begin{align*}
\left\{ B_{at} \, : \, t \geq 0 \right\} \stackrel{d}{=} \left\{ a^H B_t \, : \, t \geq 0 \right\},
\end{align*}
where $X \stackrel{d}{=} Y$ means that two processes $X$ and $Y$ have the same distribution.
Moreover, $B$ has stationary increments, that is, for every $s \geq 0$ ,
\begin{align*}
\left\{ B_{t+s} - B_s \, : \, t \geq 0 \right\} \stackrel{d}{=} \left\{B_t \, : \, t \geq 0 \right\}.
\end{align*}

\medskip

This article is concerned with estimating the size of the level sets of $B$, which are defined for any $x\in\mathbb{R}$ as
\begin{align}
\label{DefLS}
\mathcal{L}_x = \left\{ t \geq 0 \, : \, B_t=x\right\}.
\end{align} 
\par This line of research started with the seminal work of Taylor and Wendel \cite{taylor1966}, who were the first to study the Hausdorff dimensions of the level sets (and of the graph) in the case of a standard Brownian motion. They proved among other things that, for any fixed $x\in\mathbb{R}$, each Brownian level set $\mathcal{L}_x$ has a Hausdorff dimension $\frac{1}{2}$ with probability one. Their results were extended later on by Perkins \cite{perkins1981} who showed that, with probability one, the level sets $\mathcal{L}_x$ have a Hausdorff dimension $\frac{1}{2}$  for all $x\in\mathbb{R}$. Hence, the local structure of the level sets in the Brownian case is well understood.

Another method to describe the geometric properties of the single paths of a given process is in terms of its sojourn times. Here, the goal is to study the dimension of the amount of time spent by the stochastic process inside a moving boundary, that is, of the form
\begin{align*}
E(\phi) := \left\{ t \geq 0 \, : \, |B_t| \leq \phi(t)\right\},
\end{align*}  
where $\phi:\mathbb{R}_+\to\mathbb{R}$ is an appropriate function.

Strongly related to our note, we mention the recent work of Nourdin, Peccati and Seuret \cite{nourdin2018sojourn}, in which a specific large scale dimension is computed for the sojourn times
\begin{align}
\label{sojournTime}
E_\gamma := \left\{ t \geq 0\, : \, |B_t| \leq t^\gamma \right\}, \quad 0<\gamma< H,
\end{align}
of the fractional Brownian motion $B$. 
Note that this choice for $\phi$ is completely 
natural here because, on the one hand, the fractional Brownian motion is selfsimilar (hence the choice of a power function for $\phi$) and, on the other hand, it satisfies a law of iterated logarithm as $t\to\infty$ (hence the range $(0,H)$ for $\gamma$). Actually, \cite{nourdin2018sojourn} extended
to the fractional Brownian motion the results given by Seuret and Yang \cite{seuret2019sojourn} 
in the framework of the standard Brownian case.

In general, defining a notion of fractal dimension for a subset of $\mathbb{R}^d$ involves taking into consideration the microscopic (i.e. local) properties of this set. However, many models in statistical physics are based on the Euclidean lattice $\mathbb{Z}^d$; in this case, it may look more natural to rely on the macroscopic (i.e. global) properties of the set to define a notion of dimension. 
This is what Barlow and Taylor proposed in \cite{barlow1989, barlow1992}. Their 
dimension, called {\it macroscopic Hausdorff dimension}, has proven to be relevant in many contexts.
This is the one that was used in  \cite{nourdin2018sojourn,seuret2019sojourn}, and also the one we will
use in the present note, because it can give a good intuition about the geometry of the set into consideration, precisely whether it is scattered or not. Precise definitions will be given in Section \ref{MHDsubsection}. At this stage, we only mention that we denote this macroscopic Hausdorff dimension by $\Dim$.

\medskip

Our note can be considered as an addendum to \cite{nourdin2018sojourn}. 
Let $\mathcal{L}_x$ be the level sets  associated with a fractional Brownian motion.
In  \cite{nourdin2018sojourn}, the following is shown.
\begin{theorem}\label{thm1}
	\normalfont
Fix $x \in \mathbb{R}$. Then 
\begin{align*}
\mathbb{P}(\Dim \mathcal{L}_x= 1-H)=1.
\end{align*}
\end{theorem}
Our aim is to extend Theorem \ref{thm1} from ``$\forall x$, $\mathbb{P}(\ldots)=1$''
to ``$\mathbb{P}(\forall x: \ldots)=1$''. 
To this end, new and non-trivial arguments are required. We will prove the following.
\begin{theorem}
	\normalfont
\label{LevelSets}
\begin{align}
\mathbb{P}(\forall x\in\mathbb{R}:\,\Dim \mathcal{L}_x = 1-H)=1.
\end{align} 
\end{theorem} 

We note that our Theorem \ref{LevelSets} also recovers Seuret-Yang's result \cite[Theorem 2]{seuret2019sojourn} (Brownian motion), with what we believe is a more natural proof. 

\medskip

Throughout all the note, every random object is defined on a common probability space $\left(\Omega , \mathcal{A}, \mathbb{P}\right)$, and $\mathbb{E}$ denotes the expectation with respect to $\mathbb{P}$.
 
\section{Preliminaries}
\label{Preliminaries}

This section gathers the different tools that will be needed in order to prove Theorem \ref{LevelSets}.


\subsection{Macroscopic Hausdorff Dimension}
\label{MHDsubsection}

Following the notations of \cite{khoshnevisan2017intermittency, khoshnevisan2017macroscopic}, we consider the intervals $S_{-1} =[0,1/2)$ and $S_n = [2^{n-1},2^n)$ for $ n \geq 0$. For $E \subset \mathbb{R}^+$,  we define the set of {\it proper covers} of $E$ restricted to $S_n$ by 
\begin{align*}
\mathcal{I}_n(E) = \left\{ 
\begin{matrix}
\left\{I_i\right\}_{i=1}^{m}\, : & I_i=[x_i,y_i] \, \mbox{with} \, x_i,y_i \in \mathbb{N}, \, y_i>x_i,  \\ & I_i \subset S_n \, \mbox{ and } \, E \cap S_n \subset \bigcup_{i=1}^{m} I_i.
\end{matrix} \right\}
\end{align*}
For any set $E \subset \mathbb{R}^+$, $\rho \geq 0$ and $n \geq -1$, we define
\begin{align}
\label{nu^n_rho}
\nu_{\rho}^{n}(E) = \inf \left\{ \sum_{i=1}^{m} \left(\dfrac{\diam(I_i)}{2^n }\right )^\rho : \, \left\{I_i\right\}_{i=1}^{m} \in \mathcal{I}_n(E)\right \},
\end{align}
where $\diam([a,b])=b-a$.

The key point in the definition of $\nu_{\rho}^{n}(E)$  is that the sets $I_i$  are non-trivial intervals with {\it integer} boundaries; in particular, the infimum is reached. 
\begin{definition}
Let $E \subset \mathbb{R}^+$. The macroscopic Hausdorff dimension of $E$ is defined by 
\begin{align}
\Dim E = \inf  \left\{ \rho >0 : \, \sum_{n \geq -1} \nu_{\rho}^{n}(E) < + \infty \right \}.
\label{DefDim}
\end{align}
\end{definition}
We observe that $\Dim E$ always belongs to $[0,1]$, whatever $E \subset \mathbb{R}^+$. Indeed, 
consider the family $I_i=[2^{n-1}+i-1,2^{n-1}+i]$, $1\leq i\leq 2^{n-1}$, which belongs to $\mathcal{I}_n(E)$ and satisfies $\sum_{i=1}^{m} \left(\dfrac{\diam(I_i)}{2^n }\right )^\rho=\frac12 2^{n(1-\rho)}$.
Thus, $\nu^n_{1+\varepsilon}(E)\leq 2^{-n\varepsilon}$ for all $\varepsilon>0$, implying in turn that $\Dim E\leq 1+\varepsilon$ for all $\varepsilon>0$. As a result, we have that
$\Dim E\in[0,1]$.

In (\ref{nu^n_rho}), the covers are chosen to have length larger than 1. This shows that the macroscopic Hausdorff dimension does not rely on the local structure of the underlying set.

The dimension of a set is unchanged when one removes any bounded subset, since the series in (\ref{DefDim}) converges if and only if its tail series converges. 
Consequently, the dimension of any bounded set $E$ is zero. But the converse is not true, for example $\Dim (\{ 2^n,\,n\geq 1\})=0$.

The macroscopic Hausdorff dimension not only counts the number of covers of a set but also it gives an intuition about the geometry of the set. Precisely, the more the points of the set are scattered, the larger its dimension. For instance for $0< \alpha < 1$, define the two sets $A_\alpha$ and $B_\alpha$ by for all $n \geq 1$,
\begin{align*}
A_\alpha \cap S_n = \left\{2^{n-1}+ k \frac{2^{n-1}}{2^{n\alpha}} : k \in \{0,...,2^{n\alpha}-1 \}\right\};
\\ B_\alpha \cap S_n = \left\{2^{n-1}+ \frac{k}{2^{n\alpha}} : k\in \{0,...,2^{n\alpha}-1 \}\right\}.
\end{align*}
Even though both sets have same cardinality but $\Dim A_\alpha = \alpha$ whereas $\Dim B_\alpha = 0$.

These features make the  macroscopic Hausdorff dimension an interesting quantity describing the large scale geometry of a set; in particular, it appears to be well suited for the study of the level sets $\mathcal{L}_x$.

As we will see in our upcoming analysis, it might be sometimes wise to slightly modify the way $\Dim E$ is defined, to get a  definition that is more amenable to analysis. For this reason, let  us introduce, for any $E \subset \mathbb{R^+}$, $\rho>0$, $\xi \geq 0$, and $n \geq -1$, the quantity 
\begin{align}
 \widetilde{\nu}_{\rho, \xi}^{n}(E) = \inf  \left\{ \sum_{i=1}^{m} \left(\dfrac{\diam(I_i)}{2^n }\right )^\rho \left| \mbox{log}_2\dfrac{\diam(I_i)}{2^n }\right|^{\xi} :\left\{I_i\right\}_{i=1}^{m} \in \mathcal{I}_n(E) \right \}. 
\label{DefDim2}
\end{align}
The difference between $\nu_{\rho}^{n}(E)$ and $ \widetilde{\nu}_{\rho, \xi}^{n}(E) $  is that we introduce a logarithmic factor in the latter. This modification has actually no impact on the definition of $\Dim E$, as stated by the following lemma.
\begin{lemma}\label{lm3}
Let $\xi \geq 0$. For every set $E \subset \mathbb{R}^+$,
\begin{align}
\label{MHDmod}
\Dim E = \inf \left\{ \rho >0 : \, \sum_{n \geq -1} \widetilde{\nu}_{\rho,\xi}^{n}(E) < + \infty  \right\}.
\end{align}
\end{lemma}
\begin{proof}
Define $\tilde{d}_\xi = \inf \left\{ \rho >0 : \, \sum_{n \geq -1} \widetilde{\nu}_{\rho}^{n,\xi}(E) < + \infty  \right\}$.
For $n \geq -1$, consider $\left\{ I_i \right\}_{i=1}^{m} \in \mathcal{I}_n(E)$. As $I_i \subset S_n$, one has $\diam(I_i) \leq 2^{n-1}$, implying in turn that 
$\left| \mbox{log}_2  \dfrac{\diam (I_i)}{2^n} \right|^{\xi}\geq 1$.
Thus, $\widetilde{\nu}_{\rho,\xi}^{n}(E) \geq \nu_{\rho}^{n}(E) $ and then $\Dim E \leq \tilde{d}_\xi$. 

If $\Dim E=1$, the conclusion is straightforward. So, let us assume that $\Dim E<1$ and let us fix $\epsilon>0$ small enough and $\rho<1$ such that $ \rho > \Dim E + \epsilon$. Since the function $x \mapsto x^\epsilon \left|\log_2 x\right|^\xi$ is continuous on $(0,1]$ and tends to zero as $x$ tends to zero, it follows that there exists $c>0$ such that 
\begin{align*}
 \left|\log_2 x\right|^\xi \leq c x^{-\epsilon}, \, \forall x \in (0,1]
\end{align*}
We deduce that, for all $\left\{ I_i\right\}_{i=1}^{m} \in \mathcal{I}_n(E)$, 
\begin{align*}
\sum_{i=1}^{m} \left(\dfrac{\diam(I_i)}{2^n }\right )^\rho \left| \mbox{log}_2\dfrac{\diam(I_i)}{2^n }\right|^{\xi} \leq c\sum_{i=1}^{m} \left(\dfrac{\diam(I_i)}{2^n }\right )^{\rho-\epsilon} 
\end{align*}
By taking the infimum over all $\left\{ I_i\right\}_{i=1}^{m} \in \mathcal{I}_n(E)$ and recalling the definitions \eqref{nu^n_rho} and \eqref{DefDim2}, one deduces that $ \widetilde{\nu}_{\rho, \xi}^{n}(E) \leq  c\nu_{\rho - \epsilon}^{n}(E)$, implying in turn $\widetilde{d}_\xi \leq \rho - \epsilon$. Letting $\rho$ tend to $\Dim E + \epsilon$ yields the result.

\end{proof}



\subsection{Local Time of Fractional Brownian Motion}

As we will see, the use of the local time will play a key role throughout the proof of Theorem \ref{LevelSets}.

Provided it exists, the local time $x\mapsto L_t^x$ of a given process $(X_t)_{t \geq 0}$ is, for each $t$, the density of the occupation measure $\mu_t(A) = \mbox{Leb} \left\{s \in [0, t]\, : \, X_s \in A \right\}$ associated with $X$; otherwise stated, one has $L_t = \frac{d \mu_t}{d\mbox{Leb}}$. In what follows, we shall also freely use the notation $L_t([a,b])$ to indicate the quantity $L_t(b)-L_t(a)$.

The case where $X$ is Gaussian (and centered, say) has been widely studied in the literature. 
For instance, we can refer to the survey by Dozzi \cite{dozzi2003occupation}.
One of the main striking results in the Gaussian framework is the following easy-to-check condition that ensures that $(L_t^x)_{t\in [0,T],x\in\mathbb{R}}$  exists in $L^2(\Omega)$ :
\begin{equation}
\label{LemmaBddI}
I:= \int \int_{[0,T]^2} \dfrac{ds \, dt}{\sqrt{R(s,s)R(t,t)-R(s,t)^2}} < + \infty,
\end{equation}
where $R(s,t)=\mathbb{E}\left(X_s X_t\right)$; morever, in this case we have the Fourier type representation:
\begin{align}
L^{x}_{t} = \dfrac{1}{2\pi} \int_{\mathbb{R}} dy \, e^{-iyx} \int_{0}^{t} du \, e^{iyB_u}.
\label{deflt}
\end{align}

If $X$ is Gaussian, selfsimilar of index $H$ and satisfies (\ref{LemmaBddI}), then it is immediate from (\ref{deflt}) that its local time at level $x$ also have some selfsimilarity properties in time with index $1-H$, but with a different level as stated below. More precisely, one has, for every $c>0$:
\begin{align}
\label{ss}
(L^{x}_{ct})_{t\geq 0, x\in\mathbb{R}} 
\stackrel{d}{=} 
c^{1-H} (L^{c^{-H}x}_{t})_{t\geq 0, x\in\mathbb{R}}.
\end{align}

When $X$ stands for the fractional Brownian motion $B$ of  Hurst index $H\in(0,1)$, it is immediate that (\ref{LemmaBddI}) and (\ref{ss}) are satisfied.
But we can go further. A consequence of Berman's work \cite{berman1973local} is that the local time associated to $B$ is  $\beta-$H\"older continuous in $t$ for every $\beta \leq 1-H$ and uniformly in $x$. 
On their side, German and Horowitz (see \cite[Theorem 26.1]{german1980Occupation}) proved that, for all fixed $t$, the local time $(L_{t}^{x})_{x \in \mathbb{R}}$ admits the H\"older regularity in space  stated in the following lemma.

\begin{lemma}[Spatial H$\ddot{\mbox{o}}$lder continuity of local time]
Assume $X$ is a fractional Brownian motion of Hurst index $H\in(0,1)$ and 
consider its local time $(L_{t}^{x})_{x \in K}$, where $K$ is a given compact interval in $\mathbb{R}$. 
Then, for all $\beta \in\big(0, \frac{1}{2} \left( \frac{1}{H} - 1 \right)\big)$ and for all $t \geq 0$, 
\begin{align}
\label{defC}
\mathbb{P}\left(\sup_{x,y \in K} \dfrac{\left| L_t([x,y])\right|}{|x-y|^\beta} < \infty\right)=1.
\end{align}
\label{lemma3}
\end{lemma}

As we will see, Lemma \ref{lemma3} will be one of our main key tools in order to prove Lemma \ref{a} (which is one of the steps leading to the proof of Theorem \ref{LevelSets}).



\subsection{Filtration of Fractional Brownian Motion}

A last crucial property of the fractional Brownian $B$ that we will use in order to to prove Theorem \ref{LevelSets}, is that the natural filtration associated with $B$ is Brownian. 
We mean by this that  there exists a standard Brownian motion $(W_u)_{u \geq 0}$ defined on the same probability space than $B$ such that its filtration satisfies, 
for all $t > 0$,
\begin{equation}
\label{eqFBM}
\sigma\{ B_u \, : \, u \leq t \} \subset \sigma \{ W_u \, : \, u \leq t\}.
\end{equation}

Property (\ref{eqFBM}) is an immediate consequence of the Volterra representation of $B$ (see, e.g.,  \cite{Baudoin2003equivalence}).
It will be exploited together with the Blumenthal's $0-1$ law, in the end of the proof of Proposition \ref{proposition4}.

\section{Proof of Theorem \ref{LevelSets}}

\subsection{Upper bound for Dim$_H \mathcal{L}_x$}

By a theorem in \cite{nourdin2018sojourn}, for every $\gamma\in (0,H)$, a.s.
\begin{align*}
\Dim E_\gamma = 1 - H.
\end{align*}
On the other hand, observe that for a fixed $\gamma > 0$ and $x \in \mathbb{R}$, the level set $\mathcal{L}_x$ is ultimately included in $E_\gamma$. 
Indeed, 
\begin{align*}
\mathcal{L}_x \cap \left\{ t \geq |x|^{\frac1\gamma}\right\} \subset E_\gamma.
\end{align*} 
We have recalled in Section \ref{MHDsubsection} that the macroscopic Haussdorff dimension is unsensitive to the suppression of any bounded subset. 
As a result, 
 a.s. for every $x \in \mathbb{R}$, 
 $$\Dim \mathcal{L}_x = \Dim \left( \mathcal{L}_x \cap \left\{ t \geq  |x|^{\frac1\gamma}\right\} \right) \leq \Dim E_\gamma = 1 - H.$$


\subsection{Lower bound for Dim$_H \mathcal{L}_x$}
\label{LowerBound}

Recall $S_n$ from Section \ref{MHDsubsection}, and
let us  introduce the random variables 
\begin{equation}
Z_{n}^{x} = \dfrac{L^{x}\left(S_n\right)}{2^{n(1-H)}}  \quad \mbox{and} \quad F^{x}_{N} = \sum_{n=1}^{N} Z_{n}^{x}.
\label{def}
\end{equation}
The random variables $\left(Z_{n}^{x} \right)_{n \geq -1}$ are positive, so $(F_{N}^{x})_{N \geq 1}$ is non-decreasing. We denote by $F_{\infty}^{x}$ its limit, i.e. $ F_{\infty}^{x} = \sum_{n=-1}^{\infty} Z_{n}^{x} \in [0, + \infty]$.

Using (\ref{ss}), we have for all $n \geq 0$ 
\begin{equation}
\label{Ynx}
Z_{n}^{x} \stackrel{d}{=} Z_{0}^{2^{-nH}x}.
\end{equation}
We note that  similar random variables $Y_{n}^{x} = \dfrac{L^{2^{n}x}\left(S_n\right)}{2^{n(1-H)}} $ were introduced in \cite[Section 5.3]{nourdin2018sojourn}.
However, the fact that we are dealing with other space variables compared to \cite{nourdin2018sojourn} induce several differences in our proofs.
Although its statement is exactly the same than \cite[Lemma 5]{nourdin2018sojourn}, the meaning and proof in our context of the next lemma are different (albeit quite close). 
This is why we provide all the details, for the convenience of the reader.

\begin{lemma}
There exists a (deterministic) constant $K>0$ such that
\begin{equation*}
\mathbb{P}\left(\forall x\in\mathbb{R}, \forall n\geq -1:\,\widetilde{\nu}_{1-H,H}^{n}(\mathcal{L}_x) \geq K^{-1} Z_{n}^{x}\right)=1.
\end{equation*}
\label{lemma2}
\end{lemma}

\begin{proof}
Let us introduce the random variables 
\begin{equation}
\label{eqself}
A_n := \sup_{0 \leq t \leq 2^{n}} \sup_{0 \leq h \leq 2^{n-1}} \sup_{y \in \mathbb{R}} \dfrac{L^{y}\left([t,t+h]\right)}{h^{1-H}(n - \log_{2}h)^{H}},
\end{equation}
where $\log_2$ stands for the binary logarithm (base 2).
By (\ref{ss}), we have 
\begin{align}
\label{X_n}
A_n := & \sup_{0 \leq t \leq 1} \  \sup_{0 \leq h \leq 1/2}  \ \sup_{y \in \mathbb{R}} \dfrac{L^{y}\left([2^n t, 2^n(t+h)]\right)}{(2^n h)^{1-H}(- \log_{2}h)^{H}} \\
\stackrel{d}{=} & \sup_{0 \leq t \leq 1}  \  \sup_{0 \leq h \leq 1/2} \ \sup_{y \in \mathbb{R}} \dfrac{L^{y}\left([t,t+h]\right)}{h^{1-H}( - \log_{2}h)^{H}}. \nonumber
\end{align}
By \cite[Theorem 1.2]{xiao1997holder}, there exists a (deterministic) constant $K > 0$ such that 
\begin{align}
\mathbb{P}\left(\mbox{for  every } \ n \geq -1,  \ A_n \leq K\right)=1.
\label{eq1bis}
\end{align}

Now fix $x \in \mathbb{R}$, and consider the level set $\mathcal{L}_x$ defined by (\ref{DefLS}). Recall the definition (\ref{DefDim2}) of $\widetilde{\nu}^{n}_{1-H,H}(\mathcal{L}_x)$. 
If $\left(I_i= [s_i,t_i]\right)_{i=1}^{m} \in \mathcal{I}_n(\mathcal{L}_x)$ is a cover minimizing the value in (\ref{DefDim2}), we have
\begin{align}
\label{nu_1-H}
\widetilde{\nu}^{n}_{1-H,H}(\mathcal{L}_x) 
= & \sum_{i=1}^{m} \left(\dfrac{|t_i - s_i|}{2^n}\right)^{1-H} \left|\log_2 \dfrac{|t_i - s_i|}{2^n}\right|^{H}.
\end{align}
Using \eqref{X_n}-\eqref{eq1bis} with $t= \dfrac{s_i}{2^n}$, $h= \dfrac{t_i - s_i}{2^n}$, and $y = x$, we deduce that
\begin{align*}
\left(\dfrac{|t_i - s_i|}{2^n}\right)^{1-H} \left|\log_2 \dfrac{|t_i - s_i|}{2^n}\right|^{H} \geq  K^{-1}\dfrac{L^{x}(I_i)}{2^{n(1-H)}}.
\end{align*}
Back to \eqref{nu_1-H}, we get  
\begin{align*}
\widetilde{\nu}^{n}_{1-H,H}(\mathcal{L}_x) \geq
 K^{-1} \sum_{i=1}^{m} \dfrac{L^{x}(I_i)}{2^{n(1-H)}} \geq  
 K^{-1}  \dfrac{L^{x}(S_n)}{2^{n(1-H)}} =  K^{-1} Z^{x}_{n}. 
\end{align*}
where the last inequality holds because the local time $L^x_{\cdot}$ increases only on the set $I_i$ (whose union covers $\mathcal{L}_x \bigcap S_n$). This proves the claim. 
\end{proof}
Using Lemma \ref{lemma2} for the first inclusion and Lemma \ref{lm3} for the second one, we can write
$$
\{\forall x \in \mathbb{R}, \, F_{\infty}^{x} = + \infty\}\subset\{\forall x\in\mathbb{R},\,
\sum_{n\geq -1} \widetilde{\nu}^{n}_{1-H,H}(\mathcal{L}_x) =+\infty
\}\subset\{\forall x\in\mathbb{R}, \,\Dim \mathcal{L}_x \geq 1-H\}.
$$
As a consequence, we see that in order to conclude the proof of Theorem \ref{LevelSets}, it remains us to check that 
$\mathbb{P}(\forall x \in \mathbb{R}, \, F_{\infty}^{x} = + \infty ) = 1$. This is the object of the next proposition.

\begin{proposition}
\label{proposition4} We have
\begin{align}
\mathbb{P}(\forall x \in \mathbb{R}, \, F_{\infty}^{x} = + \infty ) = 1
\label{prop}
\end{align}
\end{proposition}

Note that the following weaker statement of Proposition \ref{proposition4} was shown in \cite{nourdin2018sojourn}: for all $x \in \mathbb{R}$, $\mathbb{P}(F_{\infty}^{x} = + \infty ) = 1$.
Our main contribution in the present note is precisely to prove the strongest version stated in Proposition \ref{proposition4}.


\subsection{Proof of Proposition \ref{proposition4}}

For every $a>0$, define 
\begin{align}
\widetilde{Z}^{a}_{n}= \inf_{x \in[-a,a]} Z^{x}_{n} \quad \mbox{and} \quad \widetilde{F}^{a}_{\infty}= \sum_{n \geq 1} \widetilde{Z}^{a}_{n}.
\label{defY_n^a}
\end{align}
Recalling \eqref{ss}, we get for all $n \geq 0$
\begin{align}
\label{tildeYnx}
\widetilde{Z}^{a}_{n}= \inf_{x \in[-a,a]} Z^{x}_{n} \stackrel{d}{=} \inf_{x \in[-a,a]} Z_{0}^{2^{-nH}x} = \inf_{x \in[-2^{-nH}a,2^{-nH}a]} Z_{0}^{x} = \widetilde{Z}^{2^{-nH}a}_{0}.
\end{align}
In the three forthcoming lemmas,  the following three facts are established:
\begin{enumerate}
\item[(i)] the existence of $\epsilon>0$ such that $\mathbb{P}(Z^{0}_{0}>4\epsilon) > 0$ (Lemma \ref{az}),
\item[(ii)] the existence of $a>0$ such that $ \mathbb{P}(Z_{0}^{0}>4\epsilon) \leq 2 \mathbb{P}(\widetilde{Z}^{a}_{0}>0)$ (Lemma \ref{a}),
\item[(iii)] that $\mathbb{P}\left( \widetilde{F}_{\infty}^{b} = \infty\right) \geq \mathbb{P} \left( \widetilde{Z}_{0}^{a}>0\right)$ for all $b>0$ (Lemma \ref{LemmaF^b}).
\end{enumerate}
Combining the results obtained in (i) to (iii), we deduce that 
\begin{equation}
\mathbb{P}\left( \widetilde{F}_{\infty}^{b} = \infty\right) >0\quad\mbox{for all $b>0$}.
\label{itoiii}
\end{equation}

Set $\widehat{B}_u = u^{2H}B_{1/u}$, $u>0$.
By the time inversion property of the fractional Brownian motion, $\widehat{B}$ is a fractional Brownian motion of Hurst index $H$ as well. 
We can write
\begin{align*}
L^{x}\left(S_n\right)=
 \dfrac{1}{2\pi} \int_{\mathbb{R}} dy \, e^{-iyx} \int_{2^{n-1}}^{2^n} du e^{iyu^{2H}\widehat{B}_{1/u}}.
\end{align*}
As a result, we get that $x \mapsto L^{x}\left(S_n\right)$ is $\sigma\left\{ \widehat{B}_u : u \leq 2^{-(n-1)}\right\}$-measurable, 
implying in turn that
\begin{equation}
\sigma \left\{ \widetilde{Z}^b_n \, : \, n \geq M \right\} \subset \sigma\left\{ \widehat{B}_u : u \leq 2^{-(M-1)}\right\}
\label{sigma}
\end{equation}
for every $M \geq 1$.
Consequently,
\begin{align*}
\displaystyle \left\{\widetilde{F}^b_\infty = \infty\right\} \in \bigcap_{M \geq 1}  \sigma\left\{ \widehat{B}_u : u \leq 2^{-(M-1)}\right\}.
\end{align*}
Using \eqref{eqFBM}, there exists a standard Brownian motion $(W_u)_{u \geq 0}$ defined on the same probability space such that
\begin{equation}
\label{sigmaM}
\displaystyle \left\{\widetilde{F}^b_\infty = \infty\right\} \in \bigcap_{M \geq 1} \sigma \left\{ W_u \, : \, u \leq 2^{-(M-1)}\right\}.
\end{equation}
By the Blumenthal's 0-1 law, the probability $\mathbb{P}\left(\widetilde{F}^b_\infty = \infty \right)$ is either 0 or 1. But by (\ref{itoiii}), this probability is strictly positive;
hence we conclude that 
\begin{equation}
\mathbb{P}\left( \widetilde{F}_{\infty}^{b} = \infty\right) =1\quad\mbox{for all $b>0$}.
\label{itoiiibis}
\end{equation}

For every $b>0$, one has
\begin{eqnarray*}
\label{resultb}
&&\mathbb{P} \left(\forall x \in [-b, b]:\, F_{\infty}^{x}= \infty \right)=
\mathbb{P} \left(\displaystyle \inf_{x \in [-b, b]} F_{\infty}^{x}= \infty \right) =\mathbb{P}\left(\inf_{x \in [-b, b]} \sum_{N \geq 1} Z_{N}^{x}= \infty \right) \\
&\geq& \mathbb{P} \left(\sum_{N \geq 1}\inf_{x \in [-b, b]} Z_{N}^{x} = \infty \right)= \mathbb{P}\left(\widetilde{F}^b_\infty  = \infty \right) =1.
\end{eqnarray*}
We finally conclude that 
\begin{align*}
\mathbb{P}\left(\forall x \in \mathbb{R}, \, F_{\infty}^{x}= \infty \right)= \lim_{b \rightarrow \infty} \mathbb{P}\left(\forall x \in [-b,b], \, F_{\infty}^{x}= \infty \right) = 1,
\end{align*}
which is the desired conclusion of Proposition \ref{proposition4}.

To conclude, it remains to state and prove the three lemmas mentioned in points (i) to (iii).

\begin{lemma}
\label{az}
There exists $\epsilon>0$ small enough such that $ \mathbb{P}(Z^{0}_{0}>4\epsilon) > 0$.
\end{lemma}
\begin{proof}
Using that  
$\displaystyle
L^{0}\left([\frac12,1]\right)= \dfrac{1}{2\pi} \int_{\mathbb{R}} dy  \int_{\frac12}^{1} du \, e^{iyB_u}  
$, we have
\begin{align*}
\mathbb{E}\left(L^{0}\left( \left[\frac{1}{2},1\right] \right)\right)=  \dfrac{1}{2\pi} \int_{\frac{1}{2}}^{1}  u^{-H} \, du \, \int_{\mathbb{R}}   e^{-\frac{z^2}{2}} \, dz=  \dfrac{1}{\sqrt{2\pi}} \int_{\frac{1}{2}}^{1} u^{-H} \, du >0.
\end{align*} 
As a result, $ \mathbb{P}\left(Z^{0}_{0}>0\right) = \mathbb{P}\left(L^{0}\left( \left[\frac{1}{2},1\right] \right)>0\right) >0 $, and the desired conclusion follows.
\end{proof}

\begin{lemma}
\label{a}
For every $\epsilon>0$ small enough, there exists a real number $a>0$ such that 
\begin{align*}
0< \mathbb{P}(Z_{0}^{0}>4\epsilon) \leq 2 \mathbb{P}(\widetilde{Z}^{a}_{0}>0).
\end{align*}
\end{lemma}

\begin{proof}
Let $\beta < \frac{1}{2} \left( \frac{1}{H} - 1 \right)$, $K=[-1,1]$ and $J=[\frac{1}{2},1]$. 
Set 
$$
c=c(\omega):=\sup_{x\in K\setminus\{0\}}\frac{\big| L^0(J)(\omega)-L^x(J)(\omega)\big|}{|x|^\beta}.
$$
By Lemma \ref{lemma3}, we have that $\mathbb{P}(c<\infty)=1$.

Set $\eta_\epsilon=\eta_\epsilon(\omega):= \min \left\{ \left( \dfrac{\epsilon}{c(\omega)}\right)^{1 / \beta},1\right\}$. As $[-\eta_\epsilon, \eta_\epsilon ] \subset [-1,1]$, one has
\begin{align}
\label{equation1}
\forall |x| \leq \eta_\epsilon(\omega), \,  \left| (L^{0}_{1}(\omega) -  L^{x}_{1}(\omega)) - (L^{0}_{\frac{1}{2}}(\omega) -  L^{x}_{\frac{1}{2}}(\omega))\right|\leq \epsilon.
\end{align}
By triangle inequality,
\begin{align}
\label{equation2}
\left|L^{x}_{1}-  L^{x}_{\frac{1}{2}}\right| \geq \left|L^{0}_{1} -  L^{0}_{\frac{1}{2}}\right|- \left| (L^{0}_{1} -  L^{x}_{1}) - (L^{0}_{\frac{1}{2}} -  L^{x}_{\frac{1}{2}})\right|.
\end{align}
Using \eqref{equation1} and \eqref{equation2}, we have 
\begin{align}
\left\{Z_0^0 = L^{0}_{1}-  L^{0}_{\frac{1}{2}}>4\epsilon \right\} \subset \left\{ \forall |x| \leq \eta_\epsilon(\omega), \, |L^{x}_{1} - L^{x}_{\frac{1}{2}}| \geq 3 \epsilon \right\}.
\label{equation3}
\end{align}
But $\displaystyle \left\{ \forall |x| \leq \eta_\epsilon(\omega), \, |L^{x}_{1} - L^{x}_{\frac{1}{2}}| \geq 3\epsilon)\right\} = \left\{ \inf_{x \in [-\eta_\epsilon, \eta_\epsilon]} |L^{x}_{1} - L^{x}_{\frac{1}{2}}| \geq 3\epsilon\right\}$. Recalling the definition of $\widetilde{Z}_{0}^{\eta_\epsilon}$, we deduce that
\begin{align}
\mathbb{P} \left(\widetilde{Z}_{0}^{\eta_\epsilon}>0\right) \geq \mathbb{P} \left(\widetilde{Z}_{0}^{\eta_\epsilon}>3\epsilon\right)\geq \mathbb{P}\left( Z_{0}^{0}> 4 \epsilon \right) > 0.
\label{eq1}
\end{align}

Now for all $a>0$, we have
\begin{align}
\left\{ \widetilde{Z}_{0}^{\eta_\epsilon} >0 \right\} \subset \left\{ \widetilde{Z}_{0}^{a} >0 \right\} \cup \left\{ \eta_\epsilon \leq a \right\}.
\label{eq2}
\end{align}
Since $c < \infty$ a.s., one has that $\mathbb{P}\left( c \geq M \right) \rightarrow 0$ as $M \rightarrow \infty$. We can then choose $a>0$ small enough such that 
\begin{align}
\mathbb{P} \left(\eta_\epsilon \leq a \right)= \mathbb{P}\left(c \geq \frac{\epsilon}{2a \beta} \right) \leq \frac{1}{2} \mathbb{P} \left( Z_0^0 > 4\epsilon \right).
\label{eq3}
\end{align}
Using \eqref{eq1}, \eqref{eq2} and \eqref{eq3} we deduce that
\begin{align*}
\mathbb{P} \left( Z_{0}^{0}  > 4 \epsilon \right) \leq 
 \mathbb{P} \left( \widetilde{Z}_{0}^{\eta_\epsilon} > 0 \right) \leq 
\mathbb{P} \left( \widetilde{Z}_{0}^{a} >0 \right) + \mathbb{P} \left( \eta_\epsilon \leq a \right)  \leq 
 \mathbb{P} \left( \widetilde{Z}_{0}^{a}>0\right) + \frac{1}{2} \mathbb{P}\left( Z_{0}^0 > 4 \epsilon\right).
\end{align*}
Finally, this yields 
\begin{align*}
0< \mathbb{P} \left(Z_{0}^{0}  > 4 \epsilon \right) \leq 2\mathbb{P} \left( \widetilde{Z}_{0}^{a}>0\right),
\end{align*}
which is the desired conclusion.
\end{proof}

\begin{lemma}
\label{LemmaF^b}
For any $a,b>0$, we have 
\begin{align*}
\mathbb{P}\left( \widetilde{F}_{\infty}^{b} = \infty\right) \geq \mathbb{P} \left( \widetilde{Z}_{0}^{a}>0\right).
\end{align*} 
\end{lemma}

\begin{proof}
Fix $\gamma>0$ and $a,b >0$, consider the event $A_{\gamma, b} = \left\{ \widetilde{F}_{\infty}^{b} \leq \gamma\right\}$. 
By Fubini's theorem, 
\begin{align*}
\gamma \geq \mathbb{E} \left(\mathds{1}_{A_{\gamma, b}} \widetilde{F}_{\infty}^{b}\right) = \sum_{n \geq -1} \mathbb{E}\left( \mathds{1}_{A_{\gamma, b}} \widetilde{Z}_{n}^{b} \right) = \sum_{n \geq -1} \int_{0}^{\infty} \mathbb{P} \left( A_{\gamma, b} \cap \{ \widetilde{Z}_{n}^{b} > u \} \right)du.
\end{align*}
Using $\mathbb{P}\left(A \cap B\right) \geq \left( \mathbb{P}(A) - \mathbb{P}(B^c)\right)_{+}$ where $B^c$ denotes the complement of $B$, and recalling \eqref{tildeYnx}, we deduce that 
\begin{align*}
\gamma \geq \sum_{n \geq 0} \int_{0}^{\infty} \left( \mathbb{P} \left( A_{\gamma, b} \right) - \mathbb{P} \left( \widetilde{Z}_{n}^{b}  \leq u \right) \right)_{+} \, du= \sum_{n \geq 0} \int_{0}^{\infty} \left( \mathbb{P} \left( A_{\gamma,b} \right) - \mathbb{P} \left( \widetilde{Z}_{0}^{2^{-nH}b}  \leq u \right) \right)_{+}du.
\end{align*}
There exists $M\geq1$ such that $2^{-nH}b \leq a$  for all $n \geq M$. Then, for all $n\geq M$,
\begin{align*}
\mathbb{P} \left( \widetilde{Z}_{0}^{2^{-nH}b}  \leq u \right) \leq \mathbb{P} \left( \widetilde{Z}_{0}^{a}  \leq u \right)
\end{align*}
and
\begin{align*}
\gamma \geq \sum_{n \geq M} \int_{0}^{\infty} \left( \mathbb{P} \left( A_{\gamma, b} \right) - \mathbb{P} \left( \widetilde{Z}_{0}^{a}  \leq u \right) \right)_{+}du.
\end{align*}
Since the summand does not depend on $n$ and the series is bounded by $\gamma$ and thus finite, one has  necessarily
\begin{align*}
\int_{0}^{\infty} \left( \mathbb{P} \left( A_{\gamma, b} \right) - \mathbb{P} \left( \widetilde{Z}_{0}^{a}  \leq u \right) \right)_{+}du = 0. 
\end{align*}
Hence, for almost every $u \geq 0$ and every $\gamma \geq 0$,
\begin{align}
\mathbb{P} \left( \widetilde{F}_{\infty}^{b} \leq \gamma\right) =\mathbb{P} \left( A_{\gamma, b} \right) \leq \mathbb{P} \left( \widetilde{Z}_{0}^{a}  \leq u \right).
\label{result1}
\end{align}
We know that $\mathbb{P}(\widetilde{Z}_{0}^{a}  \leq u)$ is increasing as a function of $u$. Hence, (\ref{result1}) is actually true for {\it every} $u \geq 0$ and $\gamma \geq 0$.
Hence $\mathbb{P} \left( \widetilde{F}_{\infty}^{b} > n \right) \geq \mathbb{P} \left( \widetilde{Z}_{0}^{a}  > \frac{1}{n} \right)$ for all $n \in \mathbb{N}$. One conclude that 
\begin{align*}
\mathbb{P} \left( \widetilde{F}_{\infty}^{b} = \infty \right) \geq \mathbb{P} \left( \widetilde{Z}_{0}^{a}  > 0 \right).
\end{align*}
\end{proof}

{\bf Acknowledgements}. I thank my two advisors, Ivan Nourdin and St\'ephane Seuret, for their guidance in the elaboration of this note.

\bibliography{biblio/biblio_lara} 
\bibliographystyle{abbrv}

\end{document}